\documentclass{amsart}[35pt]
\usepackage{amsmath, amsthm, amscd, amssymb, amsfonts, inputenc, tikz, pstricks}
\UseRawInputEncoding
\usetikzlibrary{shapes,arrows}

\usepackage[margin=3cm]{geometry}

\newtheorem{theorem}{Theorem}[section]
\newtheorem{lem}[theorem]{Lemma}

\newtheorem{cor}[theorem]{Corollary}
\theoremstyle{definition}
\newtheorem{defn}[theorem]{Definition}
\newtheorem{ex}[theorem]{Example}

\theoremstyle{remark}

\numberwithin{equation}{section}
\begin{document}
\title[Group-Joined-Semigroups]{
Group-Joined-Semigroups and
their structures}

\author[M.H. Hooshmand]{M.H. Hooshmand }

\address{Department of Mathematics, Shiraz Branch, Islamic Azad University, Shiraz, Iran}

\email{\tt hadi.hooshmand@gmail.com, MH.Hooshmand@iau.ac.ir}


\date{}

\begin{abstract}
Every semigroup containing an ideal subgroup is called a homogroup, and it is a grouplike if and only if
it has only one central idempotent. On the other hand,
 a class of algebraic structures covering group-$e$-semigroups $(G,\cdot,e,\odot)$
has been recently introduced. Here $(G,\cdot,e)$ is a group, $(G,\odot)$ is a semigroup and the $e$-join laws $e\odot xy=e\odot x\odot y$
and $xy\odot e=x\odot y\odot e$ hold. This paper shows
close relations among these algebraic structures and proves that every group-$e$-semigroup is a
group-$e$-homogroup. Also, we
 give some necessary and sufficient conditions for a group-$e$-semigroup to be group-$e$-grouplike. As some results of the study,
we prove several characterizations of identical group-$e$-semigroups, a class of homogroups, and
give several examples such as real $b$-group-grouplikes and the Klein group-grouplike.
\end{abstract}
\maketitle
\section{Introduction and Backgrounds}
\vskip 0.4 true cm
In 1954, G. Thierrin introduced the concept of homogroup, although earlier it was studied
under the title ``semigroups with zeroide elements", by A.H. Clifford and D.D. Miller \cite{Zeroid}. Also,
in 1961 \cite{Homo}, R.P. Hunter studied the structure of some homogroups.
If a semigroup $S$ has a minimal ideal, then it is unique and so the smallest ideal.
It is called the kernel of $S$, and if the kernel is a group, then $S$ is called a
 homogroup. Indeed, a semigroup is a homogroup if and only if it contains an ideal subgroup.
 The ideal subgroup of every homogroup $S$ is unique and its identity is a central idempotent
 and also a zero for all idempotents of $S$.
It is shown that a semigroup $S$ is a homogroup if and only if there is an element in $S$ which is
 divisible on the left and right by any element of $S$ (namely, zeroied or net), and in this case, the kernel
 consists of all such elements. For example, every finite commutative semigroup is a homogroup (see \cite{Clif, Ssemi}).
 On the other hand, the concept of ``Grouplike" has been  introduced and studied by the author \cite{Grpk}.
A homogroup is a grouplike if and only if the set of its central idempotents
is a singleton. In fact, a semigroup $\Gamma$ is a grouplike  if and only if it
contains a unique central idempotent $e$ such that
for every $x\in\Gamma$ there exists a $y\in \Gamma$ in which $xy=yx=e$ (i.e., $e$
is a solvable element that is stronger than the zeroied conception). Therefore,
a grouplike can be denoted by $(\Gamma,\cdot,e)$ where $e$ is the unique central idempotent (solvable) element.
As an application of the topic of grouplikes, all semigroups whose square is a group are characterized in \cite{S2}.
Such a semigroup is also referred to as ``class united grouplike" that is a grouplike $(\Gamma,\cdot,e)$
with the property $exy=xy$ (i.e., $e$ is a bi-identity). Indeed, we have the following equivalent conditions
for a semigroup $S$:\\
(a)  $S$ is a homogroup containing a unique central idempotent $e$
with the property $exy=xy$, for all $x,y\in S$;     \\
(b)  $S$ has the square-group property (i.e., $S^2\leq S$);    \\
(c)  There exists $e\in S$ such that $(S,\cdot,e)$ is a grouplike
with the property $exy=xy$, for all $x,y\in S$.    \\
Also, one can see many other equivalent conditions and their axiomatization in \cite{Grpk, S2}.
Hence, we can determine the structure of a class of homogroups as follows.
Recall that a magma (or groupoid) is a set $X$ equipped with a single binary operation
$\cdot :X \times X \rightarrow X$ and it is denoted by $(X,\cdot)$.
\begin{theorem}
A magma $(S,\cdot)$ is a homogroup with a unique idempotent being a bi-identity if and only if there exists a class
group $(\mathcal{G},\centerdot)$ $($i.e., a group whose elements are sets$)$ and a choice function $\varphi
:\mathcal{G}\rightarrow \cup \mathcal{G}$ such that $\Gamma=
\cup\mathcal{G}$ and $\cdot=\centerdot^\varphi$, where
$
x\centerdot^\varphi y:=\varphi(A_x\centerdot A_y)
$ and $A_x$, $A_y$ are those elements of $\mathcal{G}$ that contain $x$, $y$
respectively.
\end{theorem}
\begin{proof}
This is a direct result of the above explanations, arguments, and Theorem 2.11 of \cite{Grpk}.
\end{proof}
Recall that $Z(S)$ and $It(S)=E(S)$
are the center and the set of all idempotents of $S$,
respectively, and $Zt(S):=Z(S)\cap It(S)$ denotes the set of all
central idempotents of $S$. A subset $A$ of $S$ is said to be left $e$-unital
 if $e$ is a left (local) identity for all elements of $A$. Analogously, the right and two-sided cases are defined.
Hence, $S$ is $e$-unital if and only if it is a monoid with the identity element $e$.
Now, if $M$ is a left or right ideal submonoid (i.e., unitial ideal)  of a semigroup $S$ with the identity $e_M$,
then $M$ is a (two-sided) ideal if and only if $e_M\in Z(S)$. Moreover,
if  $M$ is an ideal submonoid of $S$, then $M=e_MS=Se_M$, hence $M$ is the least ideal
of $S$ containing $e_M$.
For if $M$ is a left ideal, $m\in M$, $x\in S$ and  $e_M\in Z(S)$, then
$$
mx=(me_M)x=m(e_Mx)=m(xe_M)=(mx)e_M\in M,
$$
hence $M$ is an ideal. Conversely, if $M$ is an ideal, then $xe_M, e_Mx\in M$ and so
$$
xe_M=e_M(xe_M)=(e_Mx)e_M=e_Mx,
$$
which means $e_M\in Z(S)$.\\
Note that if  $M$ is an ideal subgroup, then the condition ``containing $e_M$" in the above statement
is extra, and $M$ is the least ideal and also a maximal subgroup of $S$. It is the largest subgroup
if and only if $e_M$ is the only idempotent of $S$.

If a semigroup with a unique central idempotent has an ideal submonoid, then
the ideal submonoid is unique, and also it is the least ideal.
Of course,  an ideal subgroup $W$ of a homogroup is always unique,
maximal subgroup but not maximum (although it is the minimum ideal), and it is the largest subgroup containing $e_W$. Hence,
 $W$ is the largest subgroup if and only if $S$ has only one idempotent, and if and only if
$S$ is a unipotent (standard) grouplike.
The following diagram shows the inclusion relations among
such algebraic structures:
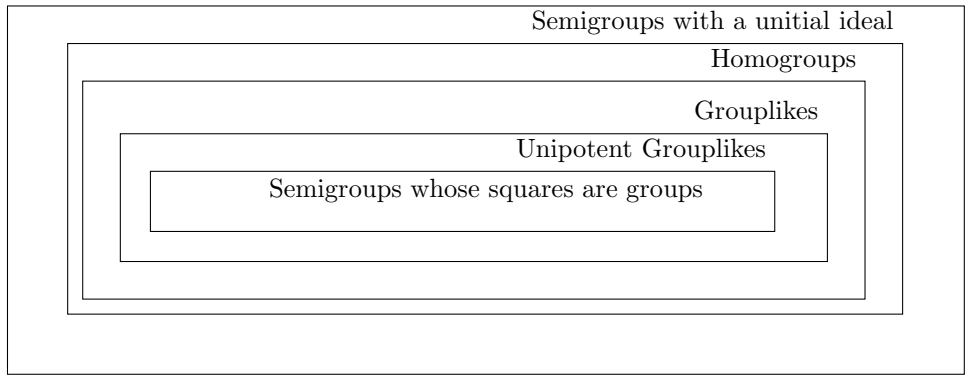
\begin{center}
\begin{tikzpicture}
\draw (-2.2,-2.2) rectangle node {Semigroups whose squares are groups} (10.52,2.7) ;
\draw (-1.4,-1.4) rectangle (9.7,2.2)  node[above left] {Semigroups with a unitial ideal};
\draw (-1.2,-1.2) rectangle (9.2,1.7)  node[above left] {Homogroups};
\draw  (-0.7,-0.7) rectangle  (8.7,1) node[above left] {Grouplikes};
\draw (-0.3,-0.3) rectangle (8,0.5)  node[above left] {Unipotent Grouplikes};
\end{tikzpicture}
\end{center}
The next example shows that the  sub-classes of semigroups in the above diagram are proper, and also
states an important class of semigroups whose square is a group.
\begin{ex} Real and Klein grouplikes are semigroups whose squares
are groups. For constructing them,
consider the additive group $(\mathbb{R},+)$ and fix $b\in
\mathbb{R}\setminus \{ 0\}$.  For each real number $a$, put
$[a]_b=b[\frac{a}{b}]$ and $(a)_b=b( \frac{a}{b})$, where $[t]$ is the integer part of $t$ and $(t):=\{t\}=t-[t]$.
Now, for every $x,y\in
\mathbb{R}$, we put
$x+_by=(x+y)_b$ and call $+_b$ {\em $b$-addition }. Then,
$(\mathbb{R},+_b)$ is a class united grouplike. Also,
$(K=\{ e, a,\eta, \alpha \},\odot)$ , introduced in Example 2.4 of \cite{Grpk} is a grouplike,
namely Klein four-grouplike.\\
For an example of a homogroup that is not a grouplike,
consider the semigroup $\mathbb{N}=\{0,1,2,  \cdots\}$ with the $\min$-binary operation. It is a homogroup
and $Zt(S)=It(S)=S$, hence $S$ is not a grouplike. It is obvious that its kernel (that is $\{0\}$)
is not the largest subgroup (only it is maximal). \\
Here, we construct a class of grouplikes that are not unipotent.
Let $\Omega$ be a set with a special element $0$ and more than two elements. For every $x,y\in \Omega$,
put $xy=x$ if $x,y\neq 0$ and $xy=0$ if $x=0$ or $y=0$. Then $\Omega$
is a grouplike that is not unipotent (standard).\\
Finally,  for an example of unipotent grouplikes that do not have the square group property,
consider the multiplicative semigroup $([0,1),\cdot)$ of (fractional) real numbers.
\end{ex}
Motivated by decomposer and associative functions on groups and semigroups
\cite{Decom}, and also $f$-multiplications and grouplikes \cite{Grpk, Fgrp}, a vast class of new algebraic structures,
namely magma-$e$- magmas, has been introduced and studied in \cite{magmag}.
A set $X$ with two binary operations ``$\cdot$", $\odot$ and containing a fixed element
$e$ is called a {\em left magma-$e$-magma}
(left $e$-magmag for short) if the following property (left $e$-join law) holds
\begin{equation}
e\odot xy=e\odot(x\odot y)\;\; ; \;\; \forall x,y\in X,
\end{equation}
where $x\cdot y$ is denoted by $xy$ and $e\odot xy=e\odot (xy)$.
Right $e$-magmag is defined in an analogous way, and $e$-magmag is
a left and right $e$-magmag, and in such case $e$ is called a joiner.\\
If the first (resp. second) magma is any type of magmas (e.g. monoids, grouplike, epigroup,  band,
homogroup,  group, etc.) then we replace the first (resp. second) magma by the type's
name, in the title.
For instance a left semigroup-$e$-semigroup (or breifly left $e$-semig) is a
left magma-$e$-magma $(S,\cdot,e,\odot)$ such that $(S,\cdot)$ and $(S,\odot)$ are semigroups.
So, $(S,\cdot,e,\odot)$ is a  left $e$-semig if
and only if $(S,\cdot)$, $(S,\odot)$ are semigroups, $e\in S$ and the left $e$-join law holds.
Very often
$e$ is considered as a special element such as left or right identity, idempotent, etc.
In this paper,  if the first semigroup of $(G,\cdot,e,\odot)$ is a
 group, then we suppose that $e$ is the same identity element of $(G,\cdot)$.
It is worth noting that for every non-trivial left $e$-semig $(S,\cdot,e,\odot)$
(i.e., $\cdot\neq \odot$), the
 second semigroup $(S,\odot)$ can not be left $e$-unital, and hence in its title
the second semigroup can not be replaced by  monoid or group.
Finally, $S$ is called left semigroup-joined-semigroup (or josemig)
if it is a left semigroup-$e$-semigroup, for all $e\in S$
(right and two-sided cases  are defined analogously), and it referred as group-joined-semigroup
if the first semigroup is a group.
\begin{ex} ({\bf The real and Klein group-grouplikes})
The structure $(\mathbb{R},+,+_b)$ is a group-joined-grouplike.
It is called {\em real $b$-group-grouplike}, and specially {\em real group-grouplike}
if $b=1$. Also, the Klein four-group $(K,\cdot)$  joined
 with the Klein four-grouplike $(K,\odot)$ forms a group-joined-grouplike.
\end{ex}
\section{Group-$e$-semigroups and its related algebraic structures}
Now, we are ready to show several closed relations among group-$e$-semigroups, homogroups
and  grouplikes.
\begin{theorem}
If  $(G,\cdot, e, \odot)$ is a left  group-$e$-semigroup, then $e\odot G$
is a subgroup of $(G,\odot)$ and the followings are equivalent\\
(a) $e\odot G$ is an ideal subgroup of $(G,\odot)$;\\
(b) $e\odot e\in Z(G,\odot)$ $($equivalently $e\odot e\in Zt(G,\odot))$;\\
(c) $(G,\odot)$ is a homogroup;\\
(d) $e\odot G$ is the largest  subgroup of $(G,\odot)$ containing $e\odot e$  and also the least ideal.
\end{theorem}
\begin{proof}
First, note that if $(S,\odot)$ is a semigroup and $e\in S$, then
$e\odot S$ is a right ideal of $S$ containing $e\odot e$. Moreover,
if $e\odot e\odot S=e\odot S$ then  $e\odot S$ is the least right ideal of $S$ containing $e\odot e$.
Also, if $e\odot S$ is a left $e\odot e$-unital sub-semigroup, then
it is the largest $e\odot e$-unital sub-semigroup
(for if $I$ is a  left $e\odot e$-unital sub-semigroup, then $I=e\odot e\odot I=e\odot I\subseteq e\odot S$).\\
Now, consider an arbitrary left $e$-semig
$(S,\cdot, e, \odot)$. It is easy to see that
$$
e\odot (xyz)=e\odot x\odot y\odot z=e\odot (xy) \odot z=e\odot x\odot(yz)
\;\; ; \;\; \forall x,y,z\in S,
$$
and so
$$
e\odot x_1\cdots x_n=e\odot x_1\odot\cdots \odot x_n \;\; ; \;\; \forall n\in \mathbb{N}\; , \; \forall x_1,\cdots,x_n\in S.
$$
Hence, \\
(i) if $e$ is a middle identity of   $(S,\cdot)$ or $(S,\odot)$ (e.g., if it is a left or right identity),
then
\begin{equation}
(e\odot x)\odot(e\odot y)=e\odot x\odot y=e\odot xy\;\; ; \;\; \forall x,y\in S.
\end{equation}
(ii) if $(S,\cdot, e, \odot)$ is a monoid-$e$-semigroup, then $(2.1)$
implies that $e\odot S$ is a sub-monoid of $(S,\odot)$ with the identity $e\odot e$ and
$e\odot e\odot S=e\odot S$. Thus $e\odot S$ is also the least right ideal containing $e\odot e$.
Moreover, if  $e\odot e\in Z(S,\odot)$, then
$$
x\odot e\odot y=x\odot e\odot e\odot y=e\odot e \odot x \odot y\in e\odot S\;\; :
\;\; \forall x,y\in S,
$$
which follows that $e\odot S$ is an ideal of $(S,\odot)$. Conversely,
if $e\odot S$ is an ideal of $(S,\odot)$, then for every $x$ there exists $z\in S$ such that
$x\odot e\odot e=e\odot z$ and so
$$
x\odot e\odot e=e\odot e\odot z=e\odot x\odot e\odot e=
e\odot x\odot e=e\odot x=e\odot e \odot x,
$$
that means $e\odot e\in Z(S,\odot)$.
\\
(iii) If $(G,\cdot, e, \odot)$ is a group-$e$-semigroup, then the argument (ii) together with
$(2.1)$ guarantee that $e\odot S$ is a subgroup of $(G,\odot)$ (because
$(e\odot x^{-1})\odot (e\odot x)=e\odot e=
(e\odot x)\odot (e\odot x^{-1})$, for every $x\in (G,\cdot)$),
and the statements (a) and (b) are equivalent. It is clear that they imply (c), and
if (c) holds, then there is an ideal subgroup of $(G,\odot)$ namely $I$.
If its identity is $i$, then $e\odot i\in I$ and so there is $j\in I$ such that
$e\odot i\odot j=i$.
Hence
 $$
 e\odot e=e\odot i^{-1}\odot i=e\odot i^{-1}\odot i\odot i=e\odot i,
 $$
 thus
 $$
 i=e\odot i\odot j=e\odot e\odot i\odot j=e\odot i=e\odot e.
 $$
 Therefore,
 $$
 I=e\odot e\odot I=e\odot I\subseteq e\odot G=i\odot G\subseteq I,
 $$
 hence $I=e\odot G$ and we arrive at (a) and also (d). Therefore, the proof is complete (since
 (d) implies (c) clearly).
\end{proof}
\begin{cor}
Every  group-$e$-semigroup is a group-$e$-homogroup. Although, there is a
left $($resp. right$)$ group-$e$-semigroup that is not a left $($resp. right$)$ group-$e$-homogroup.
\end{cor}
\begin{proof}
Since in every (two-sided) group-$e$-semigroup $(G,\cdot, e, \odot)$,
$e$ is a central element of the semigroup, then
the first part is concluded. Now, consider an arbitrary group $(G,\cdot, e)$
with $|G|>1$ and define another binary operation in $G$ by $x\odot y=x$ (resp. $x\odot y=y$).
Then, $(G,\cdot, e, \odot)$ is a  left (resp. right) group-$e$-semigroup but
$(G,\odot)$ is not a homogroup.
\end{proof}
\begin{cor}
In every left group-$e$-semigroup $(G,\cdot, e, \odot)$,
we have
\begin{equation}
(e\odot x)_{e\odot G}^{-1}=e\odot x^{-1}=e\odot (e\odot x)^{-1}\;\; ; \;\; \forall x\in G
\end{equation}
where $(e\odot x)_{e\odot G}^{-1}$ $($resp. $x^{-1}$, $(e\odot x)^{-1})$
is as the inverse in the group $e\odot G$ $($resp. $(G,\cdot))$.
\end{cor}
\begin{cor}
For every left group-$e$-semigroup $(G,\cdot, e, \odot)$ the followings are equivalent\\
(a) $(G,\cdot, e, \odot)$ is a group-$e$-grouplike;\\
(b) $e\odot e$ is an identity for the semigroup $Zt(G,\odot)$;\\
(c) $Zt(G,\odot)=\{ e\odot e\}$.\\
Also, there is a group-$e$-homogroup that is not a group-$e$-grouplike.
\end{cor}
\begin{proof}
From Theorem 2.1 it is obvious that (a) and (c) are equivalent and (c) implies (b).
Now, if (c) holds, then for every $\delta\in Zt(G,\odot)$ we have
$$
\delta=e\odot e\odot \delta=e\odot \delta^{-1}\odot \delta\odot \delta=e\odot \delta^{-1}\odot \delta=
e\odot e,
$$
and so we arrive at (c).\\
Now, fix
an integer $n>1$ and consider
the group of the least non-negative residues modulo $n$
$\mathbb{Z}_n=\{ 0, 1, 2, 3, 4, ..., n - 1\}$ (with the $n$-addition $+_n$,
that is isomorphic to the quotient group
$\mathbb{Z}/n\mathbb{Z}=\overline{\mathbb{Z}}_n=\{ \overline{0}, \overline{1}, \overline{2}, ..., \overline{n-1}
\}$). Then, $(\mathbb{Z},+_n, 0, \min)$ is a
group-$0$-homogroup that is not a group-$0$-grouplike.
\end{proof}
According the proof of Theorem 2.1, we find that some results of left group-$e$-semigroups are also valied for $e$-semigs
as  more general algebraic structures.
\begin{cor}
Let $(S,\cdot, e, \odot)$ be a left $e$-semig, then\\
(a) If $e\odot (xe)=e\odot x$ for all $x$ $($e.g., if $(S,\cdot)$ is a right $e$- unital semigroup$)$,
 then $(e\odot S,\odot)$ is a right $e\odot e$-unital sub-semigroup of $(S,\odot)$.\\
 (b) If  $(S,\cdot, e, \odot)$ is a left monoid-$e$-semigroup,
 then $(e\odot S,\odot)$ is the least right ideal of $(S,\odot)$
containing $e\odot e$  and also  the largest $e\odot e$-unitial submonoid.\\
Moreover, the followings are equivalent:\\
($b_1$) $e\odot S$ is an ideal of $(S,\odot)$,\\
($b_2$) $e\odot e\in Z(S,\odot)$,\\
($b_3$) $e\odot S$ is a left ideal of $(S,\odot)$,\\
($b_4$) $e\odot S$ is the least ideal of $(S,\odot)$ containing $e\odot e$
and also the largest $e\odot e$-unitial submonoid.
\end{cor}
{\bf $e$-Joiner maps.}
Let $(S,\cdot, e, \odot)$ be a left or right $e$-semig. Then we have
the essential maps $J_e=J^\ell_e$, $J^e=J^r_e$ from $S$ into $S$ defined by
$J_e(x)=e\odot x$, $J^e(x)=x\odot e$ (left, right  $e$-joiner map,
respectively).
A set $S$ with two associative binary operations ``$\cdot$" and $\odot$
is a left (resp. right) $e$-semig if and only if
 $J_e(xy)=J_e(x\odot y)$ (resp. $J^e(xy)=J^e(x\odot y)$), for all $x,y\in S$.
The $e$-maps have many properties and
 play important role for study and characterization of  these
algebraic structures. For example, $J_e J^e=J^eJ_e$, and $J_e$ is right $e$-periodic
(i.e. $J_e(xe)=J_e(x)$, for all $x$)
if and only if $J_eJ^e=J_e$. \\
Now, we want to show that $e$-map has the mentioned and some other properties in
$e$-semigs. This is a direct result of Lemma 2.7 of \cite{magmag}.
\begin{lem}
Assume $(S,\cdot, e, \odot)$ is a left $e$-semig. \\
(A) If $J_e$ is right $e$-periodic, then:\\
$J_e:(S,\cdot)\rightarrow (S,\odot)$ is a homomorphism;\\
$J_e:(S,\odot)\rightarrow (S,\odot)$ is an endomorphism and right canceler;\\
$J_e:(S,\cdot)\rightarrow (S,\cdot)$ is right canceler.\\
(B) If  $J_e$ is left $e$-periodic, then
$J_e$ is left canceler from the both semigroups to themselves.\\
(C) If $S$ is a left monoid-$e$-semigroup, then $J_e$ is $e$-periodic, homomorphism
from the both semigroups into the second semigroup and
\begin{equation}
\begin{array}{c}
 J_e(xy)=J_e(x\odot y)=J_e(x)\odot J_e(y)=J_e(J_e(x)y)
=J_e(xJ_e(y))\\
=J_e(J_e(x)\odot y)=J_e(x\odot J_e(y))\;\; ; \;\; \forall x,y\in S.
\end{array}
\end{equation}
\end{lem}
By using the above lemma, the next theorem characterizes the $e$-maps
of group-$e$-semigroups as functions in the group $(G,\cdot)$ (although they are defined in the semigroup
$(G,\odot)$)
and shows that they are factor projections of the
group with respect to a normal subgroup, exactly. Indeed,
for every arbitrary function $f$ from $G$ to $G$, we define (the conjugate functions)
$f^*$ and $f_*$ by the equations  $x=f^*(x)f(x)=f(x)f_*(x)$ for all $x\in G$.
If $f(x)=f(y)$ then $x=f^*(x)f(y)=f(y)f_*(x)$, but the converse
is valid if $f$ is decomposer and we have the following
definitions (see \cite{Decom}).
We call the function $f$:\\
(a)  right (resp. left) decomposer if
$$f(f^*(x)f(y))=f(y)\;\;\; \mbox{(resp. }f(f(x)f_*(y))=f(x)) \;\;\; :\;  \forall x,y\in G.$$
(b) right (resp. left) strong decomposer if
$$f(f^*(x)y)=f(y)\;\;\; (\mbox{resp. }f(xf_*(y))=f(x)) \;\;\; :\;  \forall x,y\in G.$$
(c) {\em right canceler } (resp. left canceler) if
$$f(xf(y))=f(xy)\; (\mbox{resp.}\; f(f(x)y)=f(xy))\; :\;\;\; \forall x,y\in G.$$
(d) {\em associative} if
$$f(xf(yz))=f(f(xy)z)\; :\;\;\; \forall x,y,z\in G.$$
(e) {\em strongly associative} if
$$f(xf(yz))=f(f(xy)z)=f(xyz)\; :\;\;\; \forall x,y,z\in G.$$
Note that $f$ is {\em decomposer} or {\em two-sided decomposer}
(resp. canceler) if it is left and right decomposer (resp. canceler). These functions are characterized
in terms of the projections of factor subsets of groups in \cite{Decom}. Hence, we recall the concepts as follows.\\
Let $A$ and $B$ be arbitrary subsets of a group $(G,\cdot)$. Then
the product $AB$  is called  direct and it is denoted by $A \cdot B$
if the representation of every its element by $x=ab$  with $a\in A$, $b\in B$ is unique.
By the notation $G=A \cdot B$ we mean $G=A B$ and the product $A B$ is direct
(a factorization of $G$ by two subsets). In this case, we call $A$ (resp. $B$) a left (resp. right) factor of $G$ related
to $B$ (resp. $A$).\\
If $G=A\cdot B$, then
we have the surjective maps $P_{A}:G\longrightarrow A$, $P_{B} :G\longrightarrow B$
defined by $P_{A}(x)=a$, $P_{B}(x)=b$, for every $x\in
G$ with $x=ab$, where $a\in \Delta$ and $b\in \Omega$. The maps $P_{A}$, $P_{B}$ are
called left and right projections with respect to the factorization $G=\Delta\cdot\Omega$.
Also, we call a function $f:G\longrightarrow G$ left (resp. right)
projection, if there exists a left (resp. right) factor $A$ (resp. $B$)
of $X$ such that $f=P_{A}$ (resp. $f=P_{B}$).
\\
Now, we come back to our topic. The relation $(2.3)$ implies that $J_e$ is strong decomposer, thus it satisfies the equation $f(f^*(y)x)=f(xf_*(y))=f(x)$
and we conclude that there exist $\Delta\unlhd G$ and $\Omega \subseteq G$ such that
$G=\Delta\cdot\Omega=\Delta\cdot\Omega$ and $f=P_\Omega$ (by using Theorem 3.8 of \cite{Decom}). Hence, we arrive at the next theorem.
\begin{theorem}
 If $(G,\cdot, e, \odot)$ is a left group-$e$-semigroup, then
$J_e:(G,\cdot)\rightarrow (G,\cdot)$ is strong decomposer and so
there exist $\Delta\unlhd G$ and $\Omega \subseteq G$ such that
$G=\Delta\cdot\Omega=\Delta\cdot\Omega$ and $\Delta=J_e^*(G)
=J_{e_*}(G)$, $J_e=P_\Omega$. In other
words, the left $e$-map of a left group-$e$-semigroup is a projection
of the group, for some factor subset with respect to a normal subgroup.
\end{theorem}
\begin{ex}
Consider the real $b$-group-grouplike $(\mathbb{R},+,0,+_b)$. Then,
$J_0=(\;)_b$ that is the $b$-decimal function. Now, take the Klein group-grouplike $K$ and
define $f:K\rightarrow K$ by $f(\eta)=f(e)=e$, $f(\alpha)=f(a)=a$ .
One can easily check
that $f$ is strongly associative (i.e., $f(f(xy)z)=f(xf(yz))=f(xyz)$, see[4]), $e,\eta$-periodic and idempotent
from the both semigroups to themselves. Since $(K,\cdot)$ is a group,
then $f^*(K)=f_*(K)=\{ e,\eta\}\unlhd K$, $f(K)=\{ e,a\}$ and
$(K,\cdot)=\{ e,\eta\}\cdot\{ e,a\}$ where the product is direct. Hence, we have $f=P_{\{ e,a\}}=J_e$
which agrees with the theorem and all mentioned proerties.
\end{ex}
{\bf $e$-Relation, $e$-congruence in $e$-semigs and its induced quotient semigroup.}
Recall from \cite{magmag} the following two essential  equivalence relations.
For every left or right $e$-semig $S$, we define $x \sim_e y$ (resp. $x\sim^e y$) if and
only if $e\odot x=e\odot y$ (resp. $x\odot e=y\odot e$).
It is clear that they are two equivalence
relations in $S$ and $x \sim_e y$ (resp. $x\sim^e y$) if and only if $J_e(x)=J_e(y)$ (resp. $J^e(x)=J^e(y)$).
Hence, we have the following compositions, identities and chain maps:\\
- The injective map $\lambda_e:\overline{S}=S/\sim_e\rightarrow S$ defined
by $\lambda_e(\overline{x}):=e\odot x=J_e(x)$.\\
- Since $\lambda_e:\overline{S}\rightarrow e\odot S$ is a bijection and so invertible, then
we have the bijective map $\phi_e:=\lambda^{-1}_e:e\odot S\rightarrow \overline{S}$ with
$\phi_e(e\odot x)=\overline{x}=\pi_e(x)$.\\
Therefore, denoting
$\iota_{\overline{S}}$ and $\iota_{e\odot S}$ as the identity functions
on related sets, we have the following diagrams:
\begin{equation}
S\xrightarrow{\; J_e\;}e\odot S\xrightarrow{\; \phi_e\;}  \overline{S}
\xrightarrow{\;\lambda_e\;}S\xrightarrow{\;\pi_e \;}\overline{S}
\; , \;
e\odot S\xrightarrow{\;\phi_e\;}\overline{S} \xrightarrow{\;\lambda_e \;}
S\xrightarrow{\; J_e\;}S\xrightarrow{\; \pi_e\;}  \overline{S}
\end{equation}
hence
\begin{equation}
\pi_e\lambda_e=\phi_e\lambda_e=\iota_{\overline{S}}\; , \;\lambda_e\phi_e=\iota_{e\odot S}
\; , \;\lambda_e\pi_e=J_e \; , \; \phi_eJ_e=\pi_e,
\end{equation}
and so we arrive at the following theorem and corollaries.
\begin{theorem}  If $S$ is a left $e$-semig for which $J_e$ is  right $e$-periodic
, then the two right unital semigroups
$S/\sim_e$ and $e\odot S$ are isomorphic.
\end{theorem}
\begin{cor}
If $S$ is a left monoid-$e$-semigroup,
then  $S/\sim_e\cong e\odot S$ $($as two monoids$)$.
\end{cor}
\begin{cor}  Let $G$ be a left group-$e$-semigroup,
and put $\Delta_e:=J^*_e(G)$. Then,  $\Delta_e$ is a normal subgroup of $(G,\cdot)$
and $G/\Delta_e=G/\sim_e\cong e\odot G$.
\end{cor}
There is an important real example for the above results as follows.\\
Consider the real $b$-group-grouplike $(\mathbb{R},+,0,+_b)$. We have
$$
x\sim_b y\Leftrightarrow 0+_bx=0+_b y\Leftrightarrow (x)_b=(y)_b\Leftrightarrow
x-y\in b\mathbb{Z}\Leftrightarrow x\equiv y\; \mbox{(mod $b$)}
$$
Also,
$$
\overline{x}+_b \overline{y}=\overline{x+_by}=\overline{(x+y)_b}=\overline{x+y}=
\overline{x}+\overline{y}
$$
which agrees to the previous results and so
$$
(\mathbb{R},+_b)/\sim_b=(\mathbb{R},+)/\sim_b\cong \mathbb{R}/b\mathbb{Z}
=\mathbb{R}/J^*_0(\mathbb{R})=\mathbb{R}/[\mathbb{R}]_b
$$
where $\mathbb{R}/b\mathbb{Z}$ is the quotient group $\mathbb{R}$ over the cyclic subgroup
$b\mathbb{Z}=\langle b\rangle$ and $[\mathbb{R}]_b=\{[x]_b:x\in \mathbb{R}\}$. On the other hand,
$$
0+_b\mathbb{R}=\{ (x)_b|x\in \mathbb{R}\}=b[0,1)=\mathbb{R}_b
$$
and the above theorem implies $\mathbb{R}/b\mathbb{Z}\cong (\mathbb{R}_b,+_b)$
that is the same reference $b$-bounded group (the group of
all least nonnegative (real) residues mod $b$, if $b>0$) introduced by the author
and then studied in [3].
Here, $J_e=J_0=(\;)_b=P_{b[0,1)}$ and
$$
\pi_0(x)=\overline{x}=x+b\mathbb{Z}\; , \; \lambda_0(\overline{x})=\lambda_0(x+b\mathbb{Z})
=0+_bx=(x)_b=J_0(x),
$$
$$
\phi_0(0+_bx)=\phi_0((x)_b)=x+b\mathbb{Z}=\pi_0(x).
$$
\section{Identical group-$e$-semigroups and their characterization}
In the most mentioned classes of left $e$-semigs the elememnt $e$ is a left
bi-identity of the second semigroup (i.e., $e\odot(x\odot y)=x\odot y$ for all $x,y$).
Their structures were studied in \cite{magmag}. In this section,
we construct and characterize all left (right and two-sided) identical group-$e$-semigroups,
and also their related $e$-maps will be determined, by using decomposer, associative and canceler
functions on groups (introduced and studied in \cite{Decom}).
\begin{defn}
A left $e$-semig $(S,\cdot,e,\odot)$ is called identical if
 $e$  is a left bi-identity of $(S,\odot)$, that is
\begin{equation}
e\odot xy=e\odot x\odot y=x\odot y\;\; ; \;\; \forall x,y\in S\;\;
(\mbox{left identical $e$-join law}).
\end{equation}
 \end{defn}

Therefore, a left $e$-semig $S$ is  identical  if and only if $J_e(xy)=x\odot y$ (or equivalently
$J_e(x\odot y)=x\odot y$), for all $x,y\in S$. If $S$ is a left identical $e$-semig, then the both
$e$-maps are left $e$-periodic and $J_eJ^e=J^e$. If
$e$ is a right identity of $(S,\cdot)$, then $e\in Z(S,\odot)$
and so $S$ is a (two-sided) $e$-semig.
Similar definition and discussion can be stated for the right case, and
$S$ is an identical $e$-semig if and only if it is both left and right identical $e$-semig,
i.e.,
 \begin{equation}
 e\odot xy=e\odot(x\odot y)=x\odot y=xy\odot e=(x\odot y)\odot e \;\; ; \;\; \forall x,y\in S
 \end{equation}
 Hence $e$ is a (two-sided) bi-identity for the second semigroup. \\
{\bf Note.} It is easy to see
that for monoid-$e$-semigroups, one of  the conditions left, right and two-sided identical
implies the others, and in each of the above cases $e\in Z(S,\odot)$.
\\
The real $b$-group-grouplike is an  identical group-$0$-semigroup. Indeed,
it is an identical group-$\beta$-semigroup, for every $\beta\in b\mathbb{Z}$.
For if $\beta=bk$ where $k\in \mathbb{Z}$, then
$$\beta+_b(x+y)=\beta+_bx+_by=(kb+x+y)_b=(x+y)_b=x+_by.$$
The Klien group-grouplike is also an identical group-$e$-semigroup.

\begin{theorem}  Every identical group-$e$-semigroup $(G,\cdot,e,\odot)$ is a group-joined-unipotent grouplike,
and $(G,\odot,e\odot e )$ is a class united grouplike $($a semigroup with square group property$)$.
\end{theorem}
\begin{proof}
Since $(G,\cdot,e,\odot)$ is a left (and right) identical $e$-semig, then
$$
t\odot xy=e\odot (t\odot xy)=e\odot (txy)=e\odot (t\odot x\odot y)=
t\odot x\odot y\; , \;
\mbox{(and } xy\odot t=x\odot y\odot t \mbox{)},
$$
for all $t,x,y\in S$. Hence, $(G,\cdot,e,\odot)$ is a group-joined-semigroup.
Also,  $e,e\odot e\in Z(G,\odot)$ and if $\delta$ is an $\odot$-idempotent, then
$$
e\odot e\odot\delta=e\odot\delta=e\odot\delta\odot \delta=\delta\odot \delta=\delta.
$$
Thus, $e\odot e$ is an identity for $It(G,\odot)$ and  Corollary 2.4 implies that $(G, \odot)$ is a grouplike.
Also, the property $(3.1)$ implies that the grouplike $(G,\odot)$ is class united, and hence
if $\delta\odot \delta=\delta$ then
$$
\delta=e\odot \delta \odot \delta=e\odot \delta=e\odot e\odot \delta=e\odot \delta^{-1}\odot \delta
\odot\delta=e\odot \delta^{-1}\odot \delta
=
e\odot e,
$$
that means it is unipotent.
\end{proof}
Due to the above proof,  we have the following result in more general case.
\begin{cor}
Every left identical $e$-semig is a left semigroup-joined-semigroup $($josemig$)$.
Therefore, if there exists an $e_0\in S$ such that
 $(S,\cdot,e_0,\odot)$ is a left identical $e_0$-semig, then $(S,\cdot,e,\odot)$ is a
  left $e$-semig, for every $e\in S$.
\end{cor}
Starting with an arbitrary group $(S,\cdot)$ and
a desired function $f:S\rightarrow S$, we use another binary
operation $\cdot_f$ in $S$ defined by $x\cdot_f y=f(xy)$ (namely $f$-multiplication of
``$\cdot$",  see \cite{magmag}). For example, the $b$-addition $+_b$ is a $f$-multiplication
($f$-addition) where $f=(\; )_b$.
Recall that a semigroup $(S,\cdot)$
is called surjective if the function $\cdot:S\times S\rightarrow S$ is
surjective (i.e. $SS=S$).
It is obvious that if $(S,\cdot)$ is left or right unital (e.g., all monids and groups) then
$S$ is surjective. Also, note that
if $(S,\cdot)$ is surjective,
then $\cdot_f=\cdot_g$ if and only if $f=g$ (indeed,
$\cdot_f=\cdot_g$ if and only if $f|_{SS}=g|_{SS}$).
By using $f$-multiplications, we construct and characterizes
(left, right and two-sided) identical group-e-semigroups
in several different ways and gives the general form of their second binary operation
in the next theorem.
\begin{theorem} $($Construction and characterization of identical group-e-semigroups$)$\\
 Let $(G,\cdot,e)$ be a group and $\odot$ another  binary operation in $G$. Then,
 $(G,\cdot,e,\odot)$ is an identical group-e-semigroup if and only if there exists a $($unique$)$ strong decomposer
 function $f:(G,\cdot)\rightarrow (G,\cdot)$ such that $\odot=\cdot_f$.
 Hence, denoting the set of all binary operations in $G$ by $Bio(G)$ we have
 $$
 \{ \odot\in Bio(G)\; |\; \mbox{$(G,\cdot,e,\odot)$ is an identical group-e-semigroup} \}$$
  $$
  =
  \{ \cdot_f\; |\; \mbox{$f:(G,\cdot)\rightarrow (G,\cdot)$ is associative and idempotent
  } \}
  $$
  $$
  =
  \{ \odot\in Bio(G)\; |\; \mbox{$(G,\cdot,e,\odot)$ is a group-$e$-semigroup} \}\cap
  \{ \cdot_f\; |\; f:(G,\cdot)\rightarrow (G,\cdot) \}
  $$
\end{theorem}
\begin{proof}
First, note that if  $S$ is a left identical $e$-semig, then $\odot=\cdot_{J_e}$ and $J_e:(S, \cdot)\rightarrow (S,\cdot)$ is left $e$-periodic
and idempotent. Conversely, if  $(S,\cdot,e,\odot)$ is an $e$-semig such that $\odot=\cdot_{f}$, for some left $e$-periodic function
$f:(S, \cdot)\rightarrow (S,\cdot)$, then $f=J_e=f_e$ (because $f(x)=f(ex)=e\cdot_f x=e\odot x=J_e(x)$),
and so $x\odot y=x\cdot_f y=f(xy)=J_e(xy)=e\odot xy$ for all $x,y$, and
$$f^2(x)=f(ef(ex))=e\odot e\odot x=e\odot ex=e\odot x=J_e(x)=f(x),$$ hence $f$ is
idempotent and the $e$-semig is identical. Therefore, a left $e$-semig $(S,\cdot,e,\odot)$
is identical if and only if there exists a unique left $e$-periodic function
$f:(S, \cdot)\rightarrow (S,\cdot)$ such that $\odot=\cdot_{f}$. \\
On the other hand, one can construct $e$-semigs by $f$-multiplications as follows.
Suppose $(S,\cdot)$ is an arbitrary semigroup, fix an element $e\in S$
and consider $f:S\rightarrow S$.
Then,
$$
e\cdot_f(xy)=f(e(xy))\; ,\; e\cdot_f(x\cdot_f y)=f(ef(xy))
$$
Thus, $(S,\cdot,e,\cdot_f)$ is a left $e$-semig if and only if $f$
is associative (i.e., $f(f(xy)z)=f(xf(yz))$) and
satisfies the following functional equation in $(S,\cdot)$
 \begin{equation}
f(exy)=f(ef(xy))\;\; : \;\; \forall x,y\in S.
 \end{equation}
Also, $(S,\cdot,e,\cdot_f)$ is a left identical $e$-semig if and only if
 \begin{equation}
f(exy)=f(ef(xy))=f(xy)\;\; : \;\; \forall x,y\in S.
 \end{equation}
So, if $f_ef=f_e$ (resp. $f_ef=f_e=f$, where $f_e(t):=f(et)$) then $f$ satisfies
 $(3.3)$ (resp. $(3.4)$), and if $(S,\cdot)$ is surjective, then
the converse is also true.
Hence, if  $SS=S$ (e.g., if $S$ is a group) then $f$ satisfies $(3.4)$ if and only if $f$ is left $e$-periodic and idempotent
(equivalently $f_e=f=f^2$).
We emphasis that the $e$-semigs $(S,\cdot,e,\cdot_f)$
do not generate all $e$-semigs (as one can see in \cite{magmag}).\\
Now, if $(S,\cdot,e,\odot=\cdot_f)$ is a left identical $e$-semig, then putting $g=f_e$ we have
$\odot=\cdot_f=\cdot_g$ and $g$ is associative, left $e$-periodic and idempotent (equivalently
strongly associative and left $e$-periodic).
It is worth noting that $\odot=\cdot_f$ does not imply $f$ is unique ($f=J_e$) except
that $f$ is left $e$-periodic. \\
Note that if $f$ is associative and satisfies $(3.4)$, then
it is strongly associative. Because
$$
f(xyz)=f(exyz)=f(ef(xyz))=f(f(exy)z)=f(f(xy)z)=f(xf(yz)).
$$
Now one can drive the statements by
applying the above arguments of $e$-semigs for group-$e$-semigroups.
\end{proof}
A grouplike $(\Gamma, \centerdot)$ is said to be
$f$-grouplike if and only if there exists a group $(G,\cdot)$ and an associative function
 $f:G\rightarrow G$ such that $(\Gamma, \centerdot)=(G,\cdot_f)$. The next lemma says that the unipotent grouplike
 mentioned in Theorem 3.2 is, in fact, a $f$-grouplike.
\begin{cor}
$($a$)$ A left or right group-$e$-semigroup $(G,\cdot,e,\odot)$ is identical if and only if
$(G,\odot)$ is a $f$-grouplike.\\
$($b$)$ Suppose $(S,\cdot,e,\odot)$ is a left or right  monoid-$e$-semigroup. Then,
 $S$ is a left identical monoid-$e$-semigroup if and only if
$\odot$ is a $f$-multiplication $($of course  $f$ is unique, that is the same left $e$-map$)$.
\end{cor}
Now, with do attention to the proof of Theorem 3.4, one can see the following consequences for $e$-semigs.
\begin{cor}
For every left $e$-semig   $(S,\cdot,e,\odot)$ the following statements are equivalent\\
(a) $S$ is a left identical $e$-semig,\\
(b) There exists a $($unique$)$ left $e$-periodic function $f:S\rightarrow S$ such that $\odot=\cdot_f$, \\
(c) $\odot=\cdot_{J_e}$,\\
(d)  There exists a $($unique$)$ left $e$-periodic and idempotent function $f:S\rightarrow S$ such that $\odot=\cdot_f$,\\
(e)  $f=J_e$ is the only  function from $S$ to $S$ such that $\odot=\cdot_f$.\\
Hence, we can construct and
characterize all left identical $e$-semigs as follows
$$
\{ \odot\in Bio(S)\; |\; \mbox{$(S,\cdot,e,\odot)$ is a left identical $e$-semig} \}
$$
 $$
=\{\cdot_{f}\; |\; \mbox{$f:(S,\cdot)\rightarrow (S,\cdot)$ is strongly associative and left $e$-periodic} \}
$$
%
 $$
  =
  \{ \odot\in Bio(S)\; |\; \mbox{$(S,\cdot,e,\odot)$ is a left  $e$-semig} \}\cap
  \{ \cdot_f\; |\; \cdot_f=\cdot_{f_e} \}.
  $$
  \end{cor}
\begin{cor}
Let $(S,\cdot)$ be a surjective semigroup with the fixed element $e$.
Then, there is a one-to-one correspondence between
  the set of all binary operations $\odot$ such that
$(S,\cdot,e,\odot)$ is left identical $e$-semig and the set of all strongly associative
and left $e$-periodic functions from $(S,\cdot)$ to $(S,\cdot)$.
\end{cor}
\begin{cor}
If $(S,\cdot)$ is  a left $e$-unital semigroup and $\odot$ another binary
operation in $S$, then
$(S,\cdot,e,\odot)$ is a left identical $e$-semig if and only if
there exists a $($unique$)$ strongly associative function $f$ from $(S,\cdot)$ to $(S,\cdot)$
 such that $\odot=\cdot_f$.
\end{cor}

\end{document}